\documentclass[12pt,a4paper,leqno]{amsart} 

\advance\textwidth by 3.5cm 
\advance\textheight by 2.5cm
\advance\oddsidemargin by -1.6cm
\advance\evensidemargin by -1.6cm  
  
\usepackage{amsmath,amsthm,amssymb}   
\usepackage{enumitem} 
\usepackage{tikz-cd}    

\theoremstyle{plain}  
\begingroup  
\newtheorem{thm}{Theorem}[section]
\newtheorem{lem}[thm]{Lemma}

\endgroup
\theoremstyle{definition}

\newtheorem{xmp}[thm]{Example}  
\newtheorem{rem}[thm]{Remark}

\newcommand\bs[1]{{\boldsymbol #1}}

\newcommand\C{{\mathbb C}}

\newcommand\F{{\mathbb F}}

\newcommand\R{{\mathbb R}}

\newcommand\Fn{{\mathbb F}_{(n)}}
\newcommand\Cn{{\mathbb C}_{(n)}}

\newcommand\cA{{\mathcal A}}

\newcommand\cC{{\mathcal C}}

\newcommand\cH{{\mathcal H}}
\newcommand\cL{{\mathcal L}}
\newcommand\cM{{\mathcal M}}
\newcommand\cR{{\mathcal R}}

\newcommand\pd{\partial}

\renewcommand\d{{\delta}}

\newcommand\g{{\gamma}}
\newcommand\G{{\Gamma}}
\newcommand\om{{\omega}}
\newcommand\Om{{\Omega}}
\newcommand\s{{\sigma}}

\newcommand\z{{\zeta}}
\newcommand\bA{{\bs A}}

\newcommand\bx{\bs x}

\newcommand\xchi{\smash{\raise 0.5ex\hbox{$\chi$}}}

\author{Brian Jefferies}
\title{Holomorphic functions on the Lie ball and their monogenic counterparts}

 \address{School of Mathematics\\ UNSW\\ NSW 2042
  AUSTRALIA}
 \email{brian.jefferies@gmail.com}

\begin{document}

\begin{abstract} 
The Cauchy integral formula in Clifford analysis allows us to associate a holomorphic function
$\tilde f:L_n\to \C$ on the Lie ball $L_n$ in $\C^n$ with its monogenic counterpart $f:B_1(0)\to \C^{n+1}$
via the formula  
$$\tilde f(z) = \int_{rS^n}G_\om(z)\bs n(\om)f(\om)\,d\mu(\om),\quad z\in rL_n,\  0 < r <1.$$
The inverse map $\tilde f\mapsto f$ is constructed here using the Cauchy-Hua formula for
the Lie ball following the work of M. Morimoto \cite{Mori2}.
\end{abstract}

\baselineskip 16pt
\maketitle

\section{Introduction}
The correspondence between monogenic functions and holomorphic functions
of severable variables is essential for the study of the spectral theory of systems
of linear operators $\bA = (A_1,\dots,A_n)$ acting in a Banach space $X$. Functions $f(\bs A)$
of the $n$-tuple $\bA$ are formed by the analogue
\begin{equation}\label{eqn:1}
f(\bA) = \int_{\pd\Om} G_\om(\bA)\bs n(\om)f(\om)\,d\mu(\om)
\end{equation}
of the Riesz-Dunford functional calculus 
$$\varphi(T) = \frac1{2\pi i}\int_C (\lambda I-T)^{-1}\varphi(\lambda) \, d\lambda$$
for a single operator $T$. The simple contour $C$ surrounds the spectrum $\s(T)$
of $T$ and is contained in the domain $U$ of the holomorphic function $\varphi:U\to\C$.
In formula (\ref{eqn:1}), $\Omega$ is an open subset $\R^{n+1}$ whose boundary $\pd\Om$ 
is an oriented $n$-manifold with surface measure $\mu$ and outward unit normal 
$\bs n(\om)$ at $\om\in \pd\Om$. The Cauchy kernel
$\om\mapsto G_\om(\bs A)$ is the analogue of the resolvent $\lambda \mapsto (\lambda I-T)^{-1}$
of the operator $T$ and in the case $n=1$ where $\R^2$ is identified with $\C$, we have
$$G_\lambda(T) = \frac1{2\pi }(\lambda I-T)^{-1}.$$
In higher dimensions, various methods are used to define the Cauchy kernel $G_\om(\bs A)$ depending
on the properties of the $n$-tuple $\bA$ of operators and the complement $\g(\bA)\subset \R^{n+1}$
of the domain of $\om \mapsto G_\om(\bs A)$ is the \textit{monogenic spectrum} of $\bA$.
The function $f$ in formula (\ref{eqn:1}) is left monogenic in a neighbourhood of $\overline \Om$ in the sense that
it satisfies a higher dimensional analogue of the Cauchy-Riemann equations.

For bounded linear operators $A_1,\dots,A_n$ acting on $X$ that do not necessarily commute with each other, a polynomial $p:\R^n\to \C$
has a unique left monogenic extension $\tilde p$ to $\R^{n+1}$ and $\tilde p(\bA):X\to X$
is the linear operator formed by taking corresponding linear combinations of \textit{symmetric} operator products. For example, if $p(x_1,x_2) = x_1x_2$,
then $\tilde p(A_1,A_2) =\frac12(A_1A_2+A_2A_1)$.
More generally if $\varphi:U \to \C$ is a holomorphic function of $n$ complex variables with a sufficiently large domain $U\subset \C^n$, then we may define the operator $\varphi(\bA)$ to be the operator
$f_\varphi(\bA)$ given by formula (\ref{eqn:1})
for the unique monogenic extension $f_\varphi$ of $\varphi$ restricted to $U\cap\R^n$, when this is a nonempty
open subset of $\R^n$.

In order for the operator theory to make sense beyond polynomial functions, we require a clear understanding of the correspondence between
a holomorphic function $\varphi$ and its monogenic counterpart $f_\varphi$. The setup established in this note
is summarised by the following diagram with respect to the open unit ball
$B_1(0) = \{x\in\R^{n+1} : |x| < 1\ \}$ in $\R^{n+1}$, $n=1,2,\dots$, and the 
\textit{Lie ball} 
$$L_k=\{\z\in\C^{k} : |\z|^2 + \sqrt{|\z|^4-|\z_1^2+\dots+\z_k^2|^2} < 1\ \} $$ 
in $\C^k$, $k=1,2,\dots$ (the Lie ball is an example of a Cartan domain and hermitian symmetric 
space\footnote{\text{https://en.wikipedia.org/wiki/Hermitian\_symmetric\_space}}):
\begin{equation}\label{diag}
\begin{tikzcd}
&(L_{n+1},\tilde f_\varphi)\ar[d, "v"]\\
(L_n,\varphi)\ar[r, "u"]\ar[ru, "w"]&(B_1(0),f_\varphi)
\end{tikzcd}
\text{ or }
\begin{tikzcd}
&\cC\cR(L_{n+1})\ar[d, "v"]\\
\cH(L_n)\ar[r, "u"]\ar[ru, "w"]&\cM(B_1(0)).
\end{tikzcd}
\end{equation}
The map $v$ is just the restriction mapping $v:\tilde f_\varphi\mapsto f_\varphi$,  $f_\varphi=\tilde f_\varphi\restriction B_1(0)$ that sends left complex regular functions on $L_{n+1}$ (the function space $\cC\cR(L_{n+1})$) to left monogenic functions defined on $B_1(0)$
(the function space $\cM(B_1(0))$). The Cauchy integral formula shows that $v$ is a bijection. We seek to
construct the map $w$ that sends a holomorphic function $\varphi:L_n\to \C$ to
its complex regular extension  $\tilde f_\varphi\in \cC\cR(L_{n+1})$. 

In the case that the monogenic spectrum
$\g(\bA)$ of the $n$-tuple $\bA$ is contained in $B_1(0)$, the operator $\varphi(\bA)$ is defined by
$\varphi(\bA)e_0 = f_\varphi(\bA)$ for each $\varphi \in \cH(L_n)$. Even if 
$$\|\bA\|=\left(\|A_1\|^2+\dots+\|A_n\|^2\right)^\frac12 < 1,$$ 
we can only conclude that $\g(\bA)\subseteq B_{1+\sqrt2}(0)$ \cite[Equation (4.9)]{J}. 

Under the spectral reality condition $\s\big(\langle \bA,\xi\rangle\big)\subset \R$, $\xi\in \R^n$ (or briefly, $\bA$ is a \textit{hyperbolic}
$n$-tuple of operators), the Cauchy kernel $G_\om(\bs A)$ is defined via a plane wave formula,
$\g(\bA) \subset \R^n \equiv \{0\}\times \R^n$ and the corresponding function space diagram
\[
\begin{tikzcd}
&(\{0\}\times\g(\bA),\tilde f_\varphi)\ar[d, "v"]\\
(\g(\bA),\varphi)\ar[r, "u"]\ar[ru, "w"]&(\{0\}\times\g(\bA),f_\varphi)
\end{tikzcd}
\text{ or }
\begin{tikzcd}
&\cC\cR(\{0\}\times\g(\bA))\ar[d, "v"]\\
\cH(\g(\bA))\ar[r, "u"]\ar[ru, "w"]&\cM(\{0\}\times\g(\bA))
\end{tikzcd}
\]
is simply obtained from the Cauchy-Kowaleski extension of $\varphi$  \cite[Sections 3.5, 4.2]{J}---
strong spectral properties of the $n$-tuple $\bA$ serve to simplify the associated function theory.

The functional calculus $f\mapsto f(\bs A)$ for the $n$-tuple $\bA$ was introduced to motivate
the function space problems considered here and we only mention it again briefly in Section 5 as an
application of the function theory represented in diagram (\ref{diag}). Further details
of the functional calculus are in \cite{J} and may also appear in a future monograph of the author.

We start in Section 2 with a brief review of Clifford analysis and complex regular functions. 
 In Section 3, an elementary calculation establishes that the map $v$ in diagram (\ref{diag})
 is a bijection. In Section 4, the Cauchy-Hua representation for functions holomorphic in $L_n$ in terms of 
their distributional boundary values is outlined and a simple calculation allows us to construct the map $u=v\circ w$
in diagram (\ref{diag}).

\section{Clifford Analysis}
A
Clifford algebra\index{Clifford!algebra|(} $\cA$ with $n$ generators
is formed by taking the smallest real or complex algebra $\cA$ with an identity
element $e_0$ such that $\R\oplus\R^n$ is embedded in $\cA$ via the
identification of $(x_0,\bs x)\in\R\oplus\R^n $ with $x_0 e_0+\bs x\in \cA$ 
and the identity
$$\bs x^2 = -|\bs x|^2e_0 = -(x_1^2+x_2^2+\cdots+x_n^2)e_0$$
holds for all $\bs x\in\R^n$. Then we arrive at the following
definition.

Let $\F$ be either the field $\R$ of real numbers or the
field $\C$ of complex numbers. The {\it Clifford algebra\index{Clifford!algebra}} 
$\Fn$\glossary{$\Fn$} over $\F$ is a 
$2^n$-dimensional algebra with unit defined as follows. Given the standard basis vectors
$e_0, e_1,\dots,e_n$ of the vector space $\F^{n+1}$, the basis vectors $e_S$ of
$\Fn$ are indexed by all finite subsets $S$ of $\{1,2,\dots,n\}$. The basis vectors are
determined by the following rules for multiplication on $\Fn$:
\begin{eqnarray*}
e_0 = 1,\qquad&&\cr 
e_j^2 = -1,\qquad&&{\rm for}\quad 1 \leq j \leq n\cr 
e_je_k =
-e_ke_j = e_{\{j,k\}},\qquad&&{\rm for}\quad 1 \leq j < k  \leq n\cr
e_{j_1}e_{j_2}\cdots e_{j_s} = e_S,\qquad&&
{\rm if}\quad 1 \leq j_1 < j_2 < \dots<
j_s
\leq n\\
&&\qquad\qquad{\rm and}\quad S=\{j_1,\dots,j_s\}.
\end{eqnarray*}
\glossary{$e_S$}
Here the identifications $e_0 =
e_\emptyset$ and $e_j = e_{\{j\}}$ for $1 \leq j \leq n$ have been made.

Suppose that $m \leq n$ are positive integers. The vector space $\R^m$ is identified with a
subspace of $\Fn$ by virtue of the embedding $(x_1,\dots,x_m)\longmapsto \sum_{j=1}^mx_je_j$. On
writing the coordinates of $x\in\R^{n+1}$ as $x = (x_0,x_1,\dots,x_n)$, the space $\R^{n+1}$ is
identified with a linear subspace of $\Fn$ with the embedding $(x_0,x_1,\dots,x_n)\longmapsto
\sum_{j=0}^nx_je_j$.

The product of two elements $u = \sum_S u_S e_S$  and $v = \sum_S v_S  e_S,
v_S
\in \F$ with coefficients $u_S \in \F$ and $v_S \in \F$ is $uv = \sum_{S,R}u_S v_R e_S e_R$. According
to the rules for multiplication,
$e_Se_R$ is $\pm1$ times a basis vector of $\Fn$. The {\it scalar part} of $u = \sum_S u_S e_S,
u_S
\in \F$ is the term $u_\emptyset$, also denoted as $u_0$.

The Clifford algebras $\R_{(0)},\R_{(1)}$ and $\R_{(2)}$ are the real, 
\glossary{$\R_{(0)}$}\glossary{$\R_{(1)}$}\glossary{$\R_{(2)}$}
complex numbers and the
quaternions\index{quaternions}, respectively. In the case of $\R_{(1)}$, the vector $e_1$ is identified with
$i$ and for $\R_{(2)}$, the basis vectors $e_1, e_2, e_1e_2$ are 
identified with $i,j,k$ respectively. 
 
The conjugate $\overline{{e_S}}$ of a basis element $e_S$ is defined so that $e_S\overline {e_S}
=
\overline {e_S}e_S = 1$. Denote the complex conjugate of a number $c \in \F$ by
$\overline c$. Then the operation of {\it conjugation\index{Clifford!conjugation}} 
$u\longmapsto\overline{u}$\glossary{$\overline{u}$} defined by
$\overline u = \sum_S \overline{u_S}\,\overline{e_S}$ for every $u = \sum_S u_S e_S, u_S \in
\F$ is an involution of the Clifford algebra $\Fn$ and $\overline{vu}=\overline{u}\,
\overline{v}$ for all elements $u$ and $v$ of $\Fn$. Because $e_j^2=-1$, the conjugate 
$\overline{e_j}$ of $e_j$ is $-e_j$.
An inner product is defined on $\Fn$ by
the formula $(u,v) = [u\overline{v}]_0 = \sum u_S\overline{v_S}$ for every $u = \sum_S u_S e_S$
and $v =
\sum_S v_S e_S$ belonging to $\Fn$. The corresponding norm is written as $|\cdot|$.

The induced norm on $\F^{n+1}\subset \Fn$ is the usual Euclidean norm
$$|x|=|(x_0,\dots,x_n))| = \sqrt{|x_0|^2+ \dots+ |x_n|^2},$$
so that $|x|^2= [x\overline x]_0$.
For $\z\in \C^{n+1}$ with $\z =\z_0e_0+\bs\z$, $\bs\z =\z_1e_1+\dots+\z_ne_n$,
we write $\overline{\z}^\C = \z_0e_0-\bs\z$ and $|\z|_\C$ for the square root of 
$[\z\overline{\z}^\C]_0 =\z_0^2+\dots+\z_n^2$ with positive real part
if $\z\overline{\z}^\C\in\C\setminus(-\infty,0]$ and $|0|_\C=0$.
For $\bs\z$, $\bs\chi\in \C^n$, $\langle\bs\z,\bs\chi\rangle = [\bs\z\overline{\bs\chi}^\C]_0
= \z_1\chi_1+\dots+ \z_n\chi_n$ defines a bilinear form on $\C^n$.

Because $x = (x_0,x_1,\dots,x_n)\in\R^{n+1}$ is identified
with the element $\sum_{j=0}^n x_je_j$ of $\R_{(n)}$, the conjugate $\overline x$
of $x$ in $\R_{(n)}$ is $x_0e_0 -x_1e_1-\cdots -x_ne_n$. A useful feature of Clifford algebras
is that a nonzero vector $x\in \R^{n+1}$ has an inverse $x^{-1}$ in the algebra
$\R_{(n)}$ (the {\it Kelvin inverse\index{Kelvin inverse}}) given by 
$$
x^{-1} = \frac{\overline{x}}{|x|^2} = \frac{x_0e_0-x_1e_1-\cdots -x_ne_n}
{x_0^2+x_1^2+\cdots +x_n^2}.
$$
The vector $x = (x_0,x_1,\dots,x_n)\in\R^{n+1}$ will often be written as  $x = x_0e_0 +\bs x$ with
$\bs x= (x_1,\dots,x_n)\in\R^n$.

The operator of
differentiation of a scalar function in the
$j$th coordinate in $\F^{n+1}$ is denoted by $\pd_j$
for $j=0,1,\dots,n$.

A continuously differentiable function $f:U \to \Cn$ with $f =
\sum_Sf_Se_S$ defined in an
open subset $U$ of $\F^{n+1}$, the functions $Df$ and $fD$ are defined by 
\begin{align*}
Df &= \sum_S\left((\pd_0
f_S)e_S+\sum_{j=1}^n(\pd_j f_S)e_je_S\right)\\
fD &= \sum_S\left((\pd_0
f_S)e_S+\sum_{j=1}^n(\pd_j f_S)e_Se_j\right).
\end{align*} 
Then $f$ is said to be 
{\it left monogenic\index{function!left monogenic|(}} in
$U$ if
$Df(x) = 0$ for all $x\in U$ and {\it right monogenic\index{function!right monogenic}} in $U$ if
$fD(x) = 0$ for all $x\in U$ in the case $\F=\R$ and {\it two-sided monogenic} for both left and right monogenic functions. In the complex case $\F=\C$, we use the
term \textit{complex regular} instead of monogenic. The collection of all such $\Cn$-valued functions is written as $\cM(U)$
in the left monogenic case $\F=\R$ and $\cC\cR(U)$ in the left complex regular case $\F=\C$. They form
 Fr\'echet spaces with the compact-open topology and pointwise vector operations.

The operator conjugate to $D$ is defined by
$$\overline D =e_0\pd_0 - \sum_{j=1}^ne_j\pd_j$$
so that the Laplace operator $\Delta =\pd_0^2+\dots+\pd_{n}^2$ in $\R^{n+1}$ has factorisations
$$\Delta = D\overline D=\overline D D.$$
Consequently, if $u$ is a harmonic function then $\overline D u$ is a left monogenic and
$u\overline D$ is right monogenic function.
\subsection{Cauchy integral formula}
Let $\sigma_n$ denote the volume
${2\pi^{\frac{n+1}{2}}/\Gamma\left(\frac{n+1}{2}\right)}$ of the unit
$n$-sphere $S^n=\{x\in \R^{n+1}: |x|=1 \}$ in $\R^{n+1}$. If $\G_{n+1}$ is the fundamental solution
of the Laplace operator in $\R^{n+1}$ satisfying $\Delta \G_{n+1}=\delta_0$, then
\begin{align*}
E(x) &= (\overline D\G_{n+1})(x)\\
&= \frac{1}{\sigma_n}\, \frac{\overline x}{|x|^{n+1}}
\end{align*}
defined for all $x\ne 0$ belonging to $\R^{n+1}$ is the fundamental solution of
the operator $D$, that is, $DE = \delta_0e_0$ in the sense of Schwartz distributions.

Let $G_\om(x) = E(\om-x)$ for all $\omega\neq x$ belonging to $\R^{n+1}$. For operator theory,
the time honoured strategy is to replace the vector $x$ by the $n$-tuple $\bA$ in a sensible way
to obtain the Cauchy kernel $G_\om(\bA)$, depending on the joint spectral properties of $\bA$.
In the case $n=1$, the equality $G_\om(x)=\frac1{2\pi}(\om-x)^{-1}$ holds.

Suppose that $f:U\to\Cn$ is a left monogenic function defined on an open subset $U$ of $\R^{n+1}$.
Let $\Omega$ be a bounded open set with $\overline \Omega \subset U$ and a boundary $\pd\Omega$
which is an orientable $n$-dimensional manifold with hypersurface measure $\mu$ and outward unit normal
$\bs n(\om)$ at $\om \in \pd\Omega$. Then
$$f(x)=\int_{\pd \Omega}G_\omega(x)\bs n(\omega)f(\omega)\,d\mu(\omega)$$
for every $x\in\Omega$ and the integral is zero otherwise.

If we define
$$G_\omega(\z) = \frac{1}{\sigma_n}\, \frac{\overline{\om-\z}^\C}{|\om-\z|_\C^{n+1}}$$
and
$\tilde f(\z)=\int_{\pd \Omega}G_\omega(\z)\bs n(\omega)f(\omega)\,d\mu(\omega)$
for every $\z\in\C^{n+1}$ with 
\begin{align}
&\g_\C(\z)=\{x\in \R^{n+1}: |x-\z|_\C^2=0\} \subset \Omega\quad\text{  ($n$ odd) or }\label{odd}\\\
&\g_\C(\z)= \{x\in \R^{n+1}: |x-\z|_\C^2\in (-\infty,0]\} \subset \Omega\quad\text{  ($n$  even)},\label{even}
 \end{align}
 then
 \begin{align*}
D\tilde f&= D\int_{\pd \Omega}G_\omega\bs n(\omega)f(\omega)\,d\mu(\omega)\\
 &= \int_{\pd \Omega}(DG_\omega)\bs n(\omega)f(\omega)\,d\mu(\omega)
 =0,
 \end{align*}
 so $\tilde f$ is a left complex regular function on the open set $\kappa(\Omega)$ of $\z\in\C^{n+1}$ where
 either conditions (\ref{odd}) or (\ref{even}) hold. In the integrals above, the function
 $\omega\mapsto G_\omega$, $\omega\in \pd\Omega$ has values in $C^\infty(\kappa(\Omega))$
 and $\cC\cR(\kappa(\Omega))$,
 facilitating the application of the differentiation operator $D$ under the integral sign and giving the zero function.
 
 Because $\kappa(\Omega)\cap \R^{n+1} = \Omega$ if $n$ is either odd or even, $\tilde f\restriction \Omega = f$
 ensuring that the map $f\mapsto \tilde f$ is a bijection from $\cM(\Omega)$ onto $\cC\cR(\kappa(\Omega))$.

With the $\C^{n+1}$-valued $n$-form
$$\bs \om(d\z)=\sum_{j=0}^ne_j d\z_0\wedge\dots\wedge\widehat{d\z_j}\wedge\dots\wedge d\z_n,$$
the Cauchy integral formula for a left complex regular function $f:U\to \Cn$  is
$$f(z) = \int_{\partial M}G_\z(z)\bs \om(d\z)f(\z),\quad  z\in M .$$

The closure $\overline M$ of the real $(n+1)$ manifold $M$ is contained in the open subset $U$
of $\C^{n+1}$ and its boundary $\partial M$ is an orientable $n$-manifold.
 
\section{The Lie Ball}
Let $v:\cC\cR(L_{n+1})\to \cM(B_1(0))$ be the mapping restricting a left complex regular function on 
the Lie ball to a left monogenic function on the Euclidean ball $B_1(0)$ in $\R^{n+1}$.
We first see that the Cauchy integral formula above provides a bijection $v^{-1}:f\mapsto \tilde f$ from $\cM(B_1(0))$ onto $\cC\cR(L_{n+1})$.

\begin{xmp} \label{prp:Lieball}Let $n=1,2,\dots$ and  $B_1(0) = \{x\in\R^{n+1} : |x| < 1\}$. Then
$\kappa( B_1(0) ) = L_{n+1}$ where
\begin{equation}\label{eqn:Lie}
L_{n+1} = \{\z\in\C^{n+1} : |\z|^2 + \sqrt{|\z|^4-||\z|_\C^2|^2} < 1\ \}
\end{equation}
is the {\it Lie ball}\/ in $\C^{n+1}$. 

Let $\z\in \C^{n+1}$ and suppose that
$\z=\xi+i\eta$ with $\xi,\eta\in\R^{n+1}$. If $\eta = 0$, then
$\g_\C(\z) = \{\xi\}$ so that $B_1(0) \subset \kappa( B_1(0) ) $. Moreover,
$$L_{n+1}\cap(\{0\}\times\C^{n}) = \{0\}\times L_{n},\quad n=1,2,\dots\,.$$

To establish the identity $\kappa( B_1(0) ) = L_{n+1}$, suppose that $\eta\neq 0$.
According to conditions (\ref{odd}) and (\ref{even}),
the set
$\g_\C(\z)$ is an $(n-1)$-dimensional ball or sphere  with radius $|\eta|$ in $\R^{n+1}$, lying
in the hyperplane with normal
$\eta$ and passing through $\xi\in\R^{n+1}$. 

Let $0 \le \angle({\xi,\eta})\le \pi$ be the angle
between $\xi$ and $\eta$ in $\R^{n+1}$, that is $\langle\xi,\eta\rangle =
|\xi|.|\eta|\cos(\angle({\xi,\eta}))$. The projection of $\xi$ onto $\{\eta\}^\perp$ has length
$|\xi|\sin(\angle({\xi,\eta}))$ and the projection of $\xi$ onto $\eta$ has length
$|\xi|\cos(\angle({\xi,\eta}))$. 

The projection of $\g_\C(\z)$ onto $\{\eta\}^\perp$
is a ball or sphere whose maximum distance from the origin is
$|\xi|\sin(\angle({\xi,\eta})) + |\eta|$ in the direction of the projection
of $\xi$ onto $\{\eta\}^\perp$. Because $\{\eta\}^\perp$ is distant 
$|\xi||\cos(\angle({\xi,\eta}))|$ from the hyperplane in $\R^{n+1}$ in which $\g_\C(\z)$
lies, the maximum distance  from the origin of points belonging to $\g_\C(\z)$ is 
$$\sqrt{|\xi|^2\cos^2(\angle({\xi,\eta})) + (|\xi|\sin(\angle({\xi,\eta})) + |\eta|)^2},$$ so
\begin{eqnarray*}
&& \{\z\in\C^{n+1}: \g_\C(\z) \subset B_1(0) \}\\
&&\quad = \{\z\in\C^{n+1}:\z=\xi+i\eta,\ \eta\neq0,\ |\xi|^2+|\eta|^2 + 
2|\xi| |\eta|\sin(\angle({\xi,\eta})) < 1 \}\cup B_1(0)\cr
&&\quad = \{\z\in\C^{n+1}:\z=\xi+i\eta, \ |\xi|^2+|\eta|^2 + 2(|\xi|^2
|\eta|^2-\langle\xi,\eta\rangle^2)^\frac12 < 1
\} \cr
&&\quad =\{\z\in\C^{n+1} : |\z|^2 + \sqrt{|\z|^4-||\z|_\C^2|^2} < 1\ \} 
\end{eqnarray*}
is a connected open set and the equality (\ref{eqn:Lie}) follows.
Consequently, any left monogenic function $f:B_1(0)\longrightarrow\C_{(n)} $
has a unique complex left monogenic extension 
$\tilde f:L_{n+1}\longrightarrow\C_{(n)} $ to the Lie ball $L_{n+1}$ in $\C^{n+1}$
\cite[Proposition 7]{Ry1}. 
\end{xmp}
 
 \begin{rem} As noted above, 
 $$|\z|^2 + \sqrt{|\z|^4-||\z|_\C^2|^2} = |\xi|^2+|\eta|^2 + 2(|\xi|^2
|\eta|^2-\langle\xi,\eta\rangle^2)^\frac12,\quad\text{for } \z=\xi+i\eta,\ \xi,\eta\in\R^n,$$
and so $\|\z\|_{L_n} = \left(|\z|^2 + \sqrt{|\z|^4-||\z|_\C^2|^2}\right)^\frac12$ is a norm on $\C^n$
satisfying $$|\xi+i\eta| \le \|\xi+i\eta\|_{L_n} \le |\xi|+|\eta|,\quad \xi,\eta\in\R^n.$$
It turns out that $(\C^n,\|\cdot\|_{L_n})$ is the projective tensor product $\R^n\hat\otimes_\pi\C$
with respect to the corresponding Euclidean norms on $\R^n$ and $\C$ \cite[Equation (1.7)]{Mori1}, \cite[Theorem 3.17)]{Mori2}.
 \end{rem}

\section{Holomorphic functions on the Lie ball}
For a continuous function $f$ on the closed unit disk $D_1=\{z\in\C:|z|\le1\ \}$
that is holomorphic inside $D_1$, the Cauchy integral formula
$$f(z) =\frac1{2\pi i}\int_{S^1}\frac1{\s - z}f(\s)\,d\s,\quad z\in\C,\ |z| < 1,$$
represents $f$ in terms of an integral of its boundary values over the unit circle 
$$S^1=\{z\in\C:|z|=1\}.$$
More generally, there is a bijective correspondence between holomorphic functions $f$
defined on the open unit disk $\{z\in\C:|z| <1\}$ and \textit{hyperfunctions} $T(e^{i\theta})$
on $S^1$ whose Fourier coefficients satisfy $c_m = 0$ for $m < 0$. The \textit{Cauchy-Hua integral formula}
$$f(z) =\left(T(e^{i\theta}),(1-e^{-i\theta}z)^{-1}\right),\quad z\in\C,\ |z| < 1,$$
represents the holomorphic function in terms of its boundary values $T(e^{i\theta})$.

Following M. Morimoto \cite{Mori1,Mori2}, we see how both the correspondence  and the
Cauchy-Hua integral formula extend from the open unit disk in $\C$
to the Lie ball $L_n$ in $\C^n$. We use the Cauchy-Hua integral formula
for $\cH(L_n)$ to contruct the map $u:\cH(L_n)\to\cM(B_1(0))$ in the diagram
(\ref{diag}) above.

The Hilbert space $L^2(S^{n-1})$ with respect to the invariant Borel probability measure $\mu_{n-1}$
on $S^{n-1}$ has an orthogonal decomposition $L^2(S^{n-1}) =\oplus_{k=0}^\infty \cH^k(S^{n-1})$
into spaces $\cH^k(S^{n-1})$ spherical harmonics of degree $k=0,1,\dots$\,. Spherical harmonics on $S^n$
are systematically treat in \cite[Chapter 2]{Mori2}. Further results related to Clifford analysis
in the Lie ball are in \cite{Somm}.

The orthogonal projection
$T_k:L^2(S^{n-1})\to \cH^k(S^{n-1})$ is given by
$$(T_kf)(\om) = N(k,n)\int_{S^{n-1}} P_{k,n}(\langle \tau,\om\rangle)f(\tau)\,d\mu_{n-1}(\tau)$$
for $N(k,n) = \dim\cH^k(S^{n-1}) = \frac{(2k+n-2)(k+n-3)!}{k!(n-2)!}$. The \textit{Legendre polynomials} $P_{k,n}$
are defined by the generating formula
\begin{equation}\label{GF}
(1-2tr+r^2)^{-\frac{n-2}2} = \sum_{k=0}^\infty\frac{n-2}{2k+n-2}N(k,n)P_{k,n}(t)r^k
\end{equation}
for $n =3,4,\dots$ and $P_{k,2}(t) = \cos(k\cos^{-1}t)$, $k=0,1,\dots$  \cite[Equations (1.22), (1.23)]{Mori1}.
Note that in the monograph \cite{Mori2}, the $n$-sphere $S^n$ in $\R^{n+1}$ is considered and the corresponding notation is $N(k,n-1)$ and $P_{k,n-1}$ for the unit sphere $S^{n-1}$
in $\R^n$---we have opted for typographical convenience in the present context.

Let $\Sigma^n =\{e^{i\theta}\omega:\theta\in\R,\ \omega\in S^{n-1}\}$ be the \textit{Lie sphere}. 
The Hilbert space $L^2(\Sigma^n)$ defined with respect to the probability measure $\pi^{-1}d\theta d\mu_{n-1}(\omega)$ has an orthogonal decomposition
$L^2(\Sigma^n) =\oplus_{(m,k)\in\Lambda}^\infty \cH^{m,k}(\Sigma^n)$,
$\Lambda =\{m\equiv k \mod2,\ k=0,1,\dots \}$ with
$$\cH^{m,k}(\Sigma^n) = \{e^{im\theta}S_k: S_k\in \cH^k(S^{n-1})\ \},\quad\text{for }(m,k)\in\Lambda .$$

The orthogonal projection $T_{m,k}:L^2(\Sigma^n) \to \cH^{m,k}(\Sigma^n)$ is given by 
$$(T_{m,k}f)(e^{i\theta}\omega) = \pi^{-1}N(k,n)\int_0^\pi \int_{S^{n-1}}
e^{im(\theta-\phi)}P_{k,n}(\langle \tau,\om\rangle)f(e^{im\phi}\tau)\,d\phi d\mu_{n-1}(\tau).$$
For $f\in L^2(\Sigma^n)$, the function $S_{m,k}f\in \cH^k(S^{n-1})$ is defined by 
$(S_{m,k}f)(\om) = (T_{m,k}f)(\om)$ for $\om\in S^{n-1}$. Then $S_{m,k}f$ is the restriction to 
$S^{n-1}$ of the complex polynomial
$$(S_{m,k}f)(z) = \pi^{-1}N(k,n)\int_0^\pi \int_{S^{n-1}}
e^{-im\phi}P_{k,n}(\langle \tau,z\rangle)f(e^{im\phi}\tau)\,d\phi d\mu_{n-1}(\tau),\quad z\in\C^n.$$
The unique harmonic homogeneous polynomial of degree $k$ equal to $S_{m,k}f$ on $S^{n-1}$ is
$$(\tilde S_{m,k}f)(z)  = |z|_\C^k(S_{m,k}f)(z/|z|_\C),\quad  z\in\C^n.$$

Let $\Lambda_+ =\{(m,k)\in\Lambda: k\le m\}$. It turns out that
\begin{equation}\label{eqn:holo}
f(z) =\sum_{(m,k)\in\Lambda_+} |z|_\C^{m-k}(\tilde S_{m,k}f)(z),\quad z\in L_n,
\end{equation}
for every holomorphic function $f:L_n\to \C$ \cite[Equation (3.34)]{Mori1}. The sum converges uniformly
on every compact subset of the Lie ball $L_n$ in $\C^n$.

The space of real analytic functions on the analytic manifold $\Sigma^n$ is denoted by $\cA(\Sigma^n)$.
It has the inductive limit topology defined by holomorphic functions on open sets containing $\Sigma^n$ 
in $\C^n$. Elements of $\cA'(\Sigma^n)$ are called \textit{hyperfunctions} and the holomorphic function given by 
(\ref{eqn:holo}) has a trace $T_f(e^{i\theta}\om) \in \cA'(\Sigma^n)$ given by
$$T_f(e^{i\theta}\om) = \sum_{(m,k)\in\Lambda_+}e^{im\theta}(S_{m,k}f)(\om) $$
with respect to the bilinear form \cite[Equation (5.18)]{Mori1}, see \cite[Theorem 5.2]{Mori1} and \cite[Theorem 5.7]{Mori2} in the case of $L_{n+1}$. Furthermore, the \textit{Cauchy-Hua} formula
\begin{equation}\label{eqn:CauchyH}
f(z) =\left\langle|\om - e^{-i\theta}z|_\C^{-n},T_f(e^{i\theta}\om)\right\rangle,\quad z\in L_n,
\end{equation}
holds with respect to the bilinear pairing $\left\langle \cA(\Sigma^n),\cA'(\Sigma^n)\right\rangle$.

\begin{lem} There exists a unique $\C^{n+1}$-valued function 
$$(e^{i\theta}\om,z)\mapsto\widetilde{ |\om - e^{-i\theta}z|_\C^{-n}},\quad \om\in S^{n-1},\ \theta\in\R,z\in L_{n+1}$$
analytic on $\Sigma_n$ and complex regular on $L_{n+1}$ such that
$$\widetilde{|\om - e^{-i\theta}z|_\C^{-n}} = |\om - e^{-i\theta}z|_\C^{-n}e_0,\quad\text{for all } z\in L_n.$$
\end{lem}

\begin{proof} Formula (\ref{GF}) is valid for $r\in\C$, $|r| < 1$, so for $x = \tau r$ with $\tau\in S^{n-1}$ and $ 0 < r < 1$, we have
\begin{align*}
|\om - e^{-i\theta}\tau r|_\C^{-n} &= (1-2e^{-i\theta}\langle\om,\tau\rangle r+e^{-2i\theta}r^2)^{-\frac{n}2}\\
&= \sum_{k=0}^\infty\frac{n}{2k+n}N(k,n+2)P_{k,n+2}(\langle\om,\tau\rangle)e^{-ik\theta}r^k
\end{align*}
By analytic continuation in $x\in\R^n$ with $|x| < 1$, we obtain
\begin{align*}
|\om - e^{-i\theta}\z|_\C^{-n} 
&= \sum_{k=0}^\infty\frac{n}{2k+n}N(k,n+2)e^{-ik\theta}|\z|_\C^k P_{k,n+2}(\langle\om,\z/|\z|_\C\rangle)
\end{align*}
for all $\z\in\C^n$ such that $\|\z\|_{L_n} < 1$.

Each of the homogeneous harmonic polynomials $\z\mapsto |\z|_\C^k P_{k,n+2}(\langle\om,\z/|\z|_\C\rangle)$,
$\z\in\C^n$, has a homogeneous  \textit{monogenic} extension $\z\mapsto W_{k,n+2}(\om;\z)$ to $\C^{n+1}$ given by 
\begin{align*}
W_{k,n+2}(\om;\z) &= {\overline D_\z}\left(\z_0|\bs\z|_\C^k P_{k,n+2}(\langle\om,\bs\z/|\bs\z|_\C\rangle)\right)\\
&= |\bs\z|_\C^k P_{k,n+2}(\langle\om,\bs\z/|\bs\z|_\C\rangle)
+\z_0{\overline D_{\bs\z}}\left(|\bs\z|_\C^k P_{k,n+2}(\langle\om,\bs\z/|\bs\z|_\C\rangle)\right) ,
\end{align*}
for $\z=\z_0e_0+\bs\z\in\C^{n+1}$ with $\z_0\neq 0$, see \cite[Theorem 14.8]{BDS}.
Then by \cite[Lemma 3.24, Theorem 2.29]{Mori2} we have
\begin{equation}\label{Wbnd}
|W_{k,n+2}(\om;\z)| \le \|\z\|_{L_{k+1}}\sup_{x\in S^n}|W_{k,n+2}(\om;x)|.
\end{equation}

On $\R^{n+1}$, we have
$$W_{k,n+2}(\om;x) = |\bx| ^k P_{k,n+2}(\langle\om,\bx/|\bx| \rangle)
+x_0{\overline D_{\bx}}\left(|\bx| ^k P_{k,n+2}(\langle\om,\bx/|\bx| \rangle)\right) $$
for $x=x_0e_0+\bx\in \R^{n+1}$,  $|x|=1$, $k=1,2,\dots$\, .
Calculating the derivative gives
$${\overline D_{\bx}} |\bx| ^k = - k\bx|\bx| ^{k-2}$$
and
\begin{align*}
{\overline D_{\bx}} P_{k,n+2}(\langle\om,\bx/|\bx| \rangle)&= P_{k,n+2}'(\langle\om,\bx/|\bx| \rangle){\overline D_{\bx}}\langle\om,\bx/|\bx| \rangle \\
&= P_{k,n+2}'(\langle\om,\bx/|\bx| \rangle)\left(-\om/|\bx| +\langle\om,\bx \rangle \bx/|\bx|^3 \right).
\end{align*}
We note that
\begin{align*}
\frac d{dt}(1-2tr+r^2)^{-\frac{n}2} &=rn(1-2tr+r^2)^{-\frac{n+2}2} \\
&=n\sum_{k=0}^\infty\frac{n+2}{2k+n+2}N(k,n+4)P_{k,n+4}(t)r^{k+1}\\
&= \sum_{k=0}^\infty\frac{n}{2k+n}N(k,n+2)P_{k,n+2}'(t)r^{k},\quad 0 < r < 1,\\
\end{align*}
from which we deduce that
$$P_{k,n+2}' = (n+2)\frac{N(k-1,n+4)}{N(k,n+2)}P_{k-1,n+4}$$
where
$$\frac{N(k-1,n+4)}{N(k,n+2)} = \frac{k(k+n)}{(n+2)(n+1)},$$
so that
$$P_{k,n+2}' = \frac{k(k+n)}{n+1}P_{k-1,n+4}.$$

According to \cite[Theorem 2.29 (ii)]{Mori2}, the bound $-1 \le P_{j,m}(t)\le 1$ holds for all $-1\le t\le 1$, $j=0,1,\dots$
and $m=2,3,\dots$, which gives
$$|W_{k,n+2}(\om;x)|\le 1 + k + 2\frac{k(k+n)}{n+1},\quad x\in S^n.$$
Because
$\sup_{x\in S^n}|W_{k,n+2}(\om;x)| = O(k^2)$ as $k\to\infty$, we obtain the convergence of
$$\widetilde{ |\om - e^{-i\theta}\z|_\C^{-n}} = \sum_{k=0}^\infty\frac{n}{2k+n}N(k,n+2)e^{-ik\theta}W_{k,n+2}(\om;\z),$$
for all $\om\in S^{n-1}$,\ $\theta\in\R$ and $\z\in L_{n+1}$ from the bound (\ref{Wbnd})
and the estimate $N(k,n+2)= O(k^{n})$ \cite[Equation (2.6)]{Mori2}. 
\end{proof}

\begin{thm} The map $w:\cH(L_{n})\to \cC\cR(L_{n+1})$ in the diagram (\ref{diag}) is given by 
$$\big(w(f)\big)(\z)=  
\left\langle \widetilde{|\om - e^{-i\theta}\z|_\C^{-n}},T_f(e^{i\theta}\om)\right\rangle,\quad \z\in L_{n+1},$$
for each holomorphic function $f:L_n\to\C$.
\end{thm}

\begin{proof} The function $\z\mapsto\widetilde{|\om - e^{-i\theta}\z|_\C^{-n}}$ is two-sided monogenic in the module $\cA(\Sigma_n)\otimes\Cn$,
so applying the continuous linear functional $T_f(e^{i\theta}\om)\in \cA'(\Sigma_n)$ ensures that $w(f)\in \cC\cR(L_{n+1})$.
The Cauchy-Hua formula gives $w(f) = f$ on $L_n$.
\end{proof}

\section{Application to Operator Theory}
Let $\bs A =(A_1,\dots,A_n)$ be an $n$-tuple of bounded linear operators acting on a Banach space $X$
and $\|\bA\|=\left(\|A_1\|^2+\dots+\|A_n\|^2\right)^\frac12$. For $\om\in \R^{n+1}$
and $|\om| >(1+\sqrt 2)\|\bs A\|$, let
$$G_\om(\bs A) = \frac1{\s_n}\sum_{k=0}^\infty\frac{(-1)^k}{k!} \langle \bs A,\nabla\rangle^k\overline D_\om \frac{1}{|\om|^{n-1}}.$$
If, for some $0 < \delta < 1$, the function $\om \mapsto G_\om(\bs A)$ is the restriction to $B_{1-\delta}(0)^c$ of a monogenic function,
also denoted by $\om\mapsto G_\om(\bs A)$,
then $\bs A =(A_1,\dots,A_n)$ has a unique holomorphic symmetric functional calculus on $L_n$. Moreover,
there exists $M_\d(n) > 0$ such that
$$\|p(\bs A)\|_{\cL(X)}\le M_\d(n) \|p\|_{L^\infty(S^{n-1})}$$
for every homogeneous polynomial $p$.

A given holomorphic function $\varphi\in \cH(L_n)$ has a monogenic counterpart $f_\varphi = u\varphi$ defined in $B_1(0) \subset\R^{n+1}$, according to diagram (\ref{diag}).
The continuous linear operator $\varphi(\bs A)\in \cL(X)$ is defined by 
$$\varphi(\bs A)e_0 = f_\varphi(\bs A) = \int_{rS^{n}}G_\om(\bs A)\bs n(\om)f_\varphi(\om)\,d\mu(\om)$$
for any $1-\d < r < 1$.

\bibliographystyle{plain}

\end{document}